\documentclass[10pt, a4paper]{amsart}
\usepackage{bm}
\usepackage{amsmath}
\usepackage{amssymb}
\usepackage{amsthm}
\usepackage{graphicx}
\usepackage{url}
\usepackage{enumerate}
\usepackage{xcolor}
\usepackage{mathrsfs} 

\usepackage{t1enc}
\usepackage{dsfont}

\newtheorem{thm}{Theorem}

\newtheorem{lem}[thm]{Lemma}

\newtheorem{cor}[thm]{Corollary}

\theoremstyle{definition}

\DeclareMathOperator{\ext}{ext} 

\DeclareMathOperator{\Gen}{Gen}
\newcommand{\cir}{\mathbb{T}^1}

\renewcommand{\phi}{\varphi}

\newcommand{\M}{\mathcal{M}}
\DeclareMathOperator{\MT}{\mathcal{M}_{\mathit T} }
\DeclareMathOperator{\Ms}{\mathcal{M}_{\mathit \sigma} }

\newcommand{\Mst}{\M_{\mathit \sigma}^{(t)}}

\DeclareMathOperator{\MTe}{\M_{\mathit T}^e}

\DeclareMathOperator{\Mse}{\M_{\mathit \sigma}^e}

\newcommand{\BB}{\mathscr{B}}

\newcommand{\bb}{\textbf{\textit{b}}}

\newcommand{\mult}{\ast}

\newcommand{\dbar}{{\bar d}}
\newcommand{\dbarom}{{\dbar}_\Omega}
\newcommand{\dbarm}{{\dbar}_\M}

\newcommand{\bmu}{{\hat \mu}}

\newcommand{\R}{\mathbb{R}}
\newcommand{\Z}{\mathbb{Z}}
\newcommand{\N}{\mathbb{N}}
\newcommand{\Q}{\mathbb{Q}}
\newcommand{\C}{\mathbb{C}}

\author{Jakub Konieczny \and Michal Kupsa \and Dominik Kwietniak}

\address[D. Kwietniak]{
Faculty of Mathematics and Computer Science, Jagiellonian University in Krak\'ow, ul. \L o\-jasiewicza 6, 30-348 Krak\'ow, Poland
\and
Institute of Mathematics, Federal University of Rio de Janeiro, Cidade
Universitaria - Ilha do Fund\~ao, Rio de Janeiro 21945-909, Brazil
}
\email{dominik.kwietniak@uj.edu.pl}
\urladdr{www.im.uj.edu.pl/DominikKwietniak/}

\address[M. Kupsa]{
Institute of Information Theory and Automation, The Academy of Sciences of the Czech Republic, Prague 8, CZ-18208}\email{kupsa@utia.cas.cz}

\address[J. Konieczny]{Mathematical Institute,
University of Oxford,
Andrew Wiles Building,
Radcliffe Observatory Quarter,
Woodstock Road,
Oxford,
OX2 6GG}
\email{jakub.konieczny@gmail.com}

\title[Arcwise connectedness]{Arcwise connectedness of the set of ergodic measures of hereditary shifts}
\date{\today}

\thanks{DK was supported by the  National Science Centre (NCN) grant 2013/08/A/ST1/00275 and partially supported by CAPES/Brazil grant
no. 88881.064927/2014-01. JK was supported by the National Science Centre (NCN) under
grant 2012/07/A/ST1/00185}

\begin{document}

\begin{abstract}  We show that the set of ergodic invariant measures of a shift space with a safe symbol (this includes all hereditary shifts) is arcwise connected when endowed with the $d$-bar metric.
As a consequence the set of ergodic measures of such a shift is also arcwise connected in the weak-star topology and the entropy function over this set attains all values in the interval between zero and the topological entropy of the shift (inclusive). The latter result is motivated by a conjecture of A.~Katok.
\end{abstract}
\maketitle

A shift space $X$ over the alphabet $\Lambda=\{0,1,\ldots,n-1\}$ for some $n\ge 2$ is  \emph{hereditary} if $x\in X$ and $y\le x$
(coordinate-\-wise) imply $y\in X$. Hereditary shifts were introduced
by Kerr and Li in \cite[p. 882]{KerrLi} and their basic properties are presented in \cite{hh4}. 
We say that $a\in\Lambda$ is a \emph{safe symbol} for a shift space
$X\subset\Lambda^\Z$ (see \cite{RS}) if for every $z\in\Lambda^\Z$ obtained by replacing some entries in $y\in X$ by $a$, we have that $z\in X$. By definition $0$ is a safe symbol for every hereditary shift.

The family of hereditary shifts includes: spacing shifts, beta shifts, bounded density shifts, $\BB$-admissible shifts; also, many examples of $\mathscr{B}$-free shifts and some shifts of finite type are hereditary.
 All classes on that list have been extensively studied, with the $\mathscr{B}$-free shifts attracting much attention recently (see Section \ref{sex:examples} for more details). In our setting of $\Z$ actions shifts with a safe symbol seem to be less important. The notion is useful in the context of higher dimensional shifts ($\Z^d$ actions with $d\ge 2$, see \cite{RS} and references therein). It is easy to find examples of shift spaces over $\{0,1,2\}$ which have $0$ as a safe symbol, but are not hereditary.

It should be no surprise that there are very few theorems applicable to all members of such a diverse family of shift spaces. Nevertheless, the main result of this note implies that there is a common feature of all hereditary shift spaces: for any hereditary shift $X$ and for every $t\ge 0$ the set of ergodic invariant measures with entropy less than or equal $t$, denoted $\Mst(X)$ and endowed with the $d$-bar metric $\dbarm$ is arcwise connected (Theorem \ref{thm:arcwise}). Our proof shows that it is enough to assume that there exists a safe symbol, and actually shows that there is a $\dbarom$-continuous arc in $X$ consisting of generic points for ergodic measures from the arc in $\Mst(X)$. This is a single orbit result very much in the spirit of \cite{WeissBook}. Here $\dbarom$ is a pseudometric on $X$ given by the upper asymptotic density of the set of indices two sequences in $X$ differ.

The $d$-bar metric $\dbarm$ induces a stronger topology than the usual weak$^*$ topology on the space of ergodic invariant measures. It follows that in the latter topological space the set of ergodic invariant measures with entropy less than or equal $t$ is also arcwise connected for every $t\ge 0$ (Corollary \ref{cor:weak-star-arc-connectedness}). The same holds for sets of ergodic measures with entropy strictly less than $t$. Furthermore, since the entropy function $h$ taking an ergodic measure $\mu$ to its metric entropy $h(\mu)$ is $\dbarm$ continuous, it has the Darboux (intermediate-value) property over every arc in $\dbarm$. Denoting the topological entropy of a shift space $X$ by $h_\text{top}(X)$ we say that $X$ has the \emph{intermediate entropy property over ergodic measures} if for every $\alpha\in[0,h_\text{top}(X)]$ there is an ergodic measure $\mu$ with $h(\mu)=\alpha$. In particular, every shift with a safe symbol has the intermediate entropy property and its set of ergodic measures is either a singleton or is uncountable (Corollary \ref{cor:entropy-interval}). 
The former case occurs if and only if $h_\text{top}(X)=0$ (Corollary \ref{cor:entropy-positive}).

Our result about the intermediate entropy property 
is motivated by the following conjecture due to A.~Katok: If $r>1$ and $F\colon M\to M$ is a $\mathcal{C}^r$  diffeomorphism of a smooth compact manifold $M$, then
for every $\alpha\in [0, h_\text{top}( F ))$, there is an ergodic measure $\mu$ such that the metric entropy of $F$ with respect to $\mu$ equals $\alpha$.
Katok proved that this is the case if $M$ is a compact surface,
where every ergodic measure of positive entropy is hyperbolic (\cite{Katok}, for detailed proof see
\cite[Theorem S.5.9]{KM}). Katok's result was extended to certain skew product
cases by Sun \cite{Sun1, Sun2}. The conjecture also holds for every ergodic linear automorphism of the torus as a result of work of Quas and Soo \cite{QS} and for some partially hyperbolic diffeomorphism by Ures \cite{Ures}.

The approach presented here is different from the methods used in  \cite{Katok, QS, Sun1, Sun2, Ures}. After publishing this paper on arXiv we learned that similar techniques were applied in \cite{Ayse}.

We also describe the construction of some examples illustrating that our main theorem may not hold without the assumption that the shift space is hereditary. By the same examples one can see that the conclusions of Corollary \ref{cor:weak-star-arc-connectedness} and Corollary \ref{cor:entropy-interval} are independent: neither of them implies the other. This is a manifestation of the well known fact that the metric entropy function $h$ on the set of invariant measures $\Ms(X)$ of a shift space $X$ endowed with the weak$^*$ topology is, in general, only upper semi-continuous.
On the other hand every Polish topological space is homeomorphic to a set of ergodic measures of some shift space endowed with the weak$^*$ topology (see \cite[Theorem 5]{D} and \cite{H}).
This suggests that there should be plenty of examples of shift spaces without the intermediate entropy property. It is obvious that this is the case if the shift space has at most countably many ergodic invariant measures. It is less obvious if the set of ergodic measures is arcwise connected in the weak$^*$ topology. Using a characterization of possible entropy functions due to Downarowicz and Serafin \cite{DS} we show that for every non-trivial Polish topological space $P$ there is a shift space $X$ whose set of ergodic measures $\Mse(X)$ endowed with the weak$^*$ topology is homeomorphic with $P$ and the entropy function has an isolated positive value (see Theorem \ref{thm:DS}). By the same technique, we show that for every uncountable Polish topological space $P$ (not necessarily connected!) there is a shift space $X$ with $\Mse(X)=P$ and a non-constant entropy function with the intermediate entropy property (see Theorem \ref{thm:DS2}).

Finally, our first proof of Theorem \ref{thm:arcwise} was based on a relative version of the Furstenberg unique ergodicity theorem (Theorem \ref{thm:Furstenberg} below).
We no longer need this to prove our main theorem, but as we hope it is a result of independent interest we attach it with a proof in the Appendix \ref{sec:appendix-A}.



\section{Definitions}\label{sec:definitions}


By a \emph{dynamical system} we mean a pair $(X,T)$, where $X$ is a compact metric space and $T\colon X\to X$ is a homeomorphism. 
If $(Y,S)$ is an another dynamical system and $\pi\colon X \to Y$ is a continuous map onto with $\pi \circ T = S \circ \pi$ then we call $\pi$ a \emph{factor map}, $Y$ a \emph{factor of $X$}, and $X$  an \emph{extension of $Y$}.


The set of all Borel probability measures on $X$ is denoted by $\M(X)$.
The usual weak$^*$ topology makes $\M(X)$ a compact metrizable space.

Let $\MT(X)$ denote the set of $T$-invariant measures in $\M(X)$. We write $\MTe(X)$ for the set of all ergodic measures in $\MT(X)$,
and $h(\mu)$ denotes the Kolmogorov-Sinai entropy of $\mu$. For each $\mu\in\MT(X)$ we call the triple $(X,T,\mu)$ a \emph{measure preserving system}.
Note that this is more restrictive than the usual definition.

A point $x \in X$ is \emph{generic} for a measure $\mu\in\MT(X)$ if for each continuous function $\phi\colon X\to\R$ we have
\[
\lim_{N\to\infty}\frac{1}{N}\sum_{j=0}^{N-1} \phi(T^j(x))=\int_X \phi\,d\mu.
\]
The set of all generic points for $\mu$ is denoted by $\Gen_T(\mu)$. 
We write $\Gen(X,T)$ for the set of points that are generic for some $T$-invariant measure.
For $x\in \Gen(X,T)$ we denote by $\bmu(x)$ the measure for which $x$ is generic.

Fix $n\ge 2$ and let $\Lambda$ be a finite set with $n$ elements, without loss of generality we assume that $\Lambda=\{0,1,\ldots,n-1\}$.
Let $\Omega=\Lambda^\Z$ be the set of all two-sided $\Lambda$-valued sequences.
We equip $\Omega$ with the product (Tikhonov) topology induced by the discrete topology on $\Lambda$.
The \emph{shift map} $\sigma\colon\Omega\to\Omega$ is given by $\sigma(x)_i=x_{i+1}$.
By a \emph{shift space} (for short \emph{a shift}) we mean any nonempty closed set $X\subset\Omega$ such that $\sigma(X)=X$. 


\section{Disjointness and spectral theory}\label{sec:spectral}

In this section we will review some basic facts from the spectral theory of Koopman operators. Since these results are classical we have not attempted to document the source of every of them (see \cite{GlasnerBook,Parry}). We need them for a proof of Theorem \ref{thm:disjoint-rotation}, which may be known among 
aficionados, but we were unable to find it in the literature.

Using the map $\psi(t)=\exp(2\pi i t)$ we identify the quotient group $\cir = \R/\Z$ equipped with addition $\bmod\,1$ and the circle, that is, the multiplicative group  $\mathbb{S}^1=\{z\in\C: |z|=1\}$.  By $\lambda$ we denote the Lebesgue measure on $\cir$ and its image (push forward) through $\psi$ on $\mathbb{S}^1$. Given $\alpha\in\R\setminus\Q$ the map $R_\alpha\colon\cir\to\cir$ given by $R_\alpha(t)=t+\alpha\bmod 1$ is called an \emph{irrational rotation} of the circle. It is well known that $\lambda$ is the unique invariant measure for $R_\alpha$.

Given a measure-preserving system $(X,T,\mu)$ we define the \emph{Koopman operator} $U_T\colon L^2(X,\mu)\to L^2(X,\mu)$ by $U_T(f)=f\circ T$. The Koopman operator is unitary, hence for each $f\in L^2$ the bi-infinite sequence $r_n=\langle U^n_T(f),f\rangle$ is a positive definite sequence of complex numbers. By Herglotz theorem the positive definite sequence $\langle U^n_T(f),f\rangle$ determines uniquely a finite non-negative Borel measure on circle $\cir$, called the \emph{spectral measure of $f$} and denoted by $\sigma_f$.

A \emph{cyclic subspace} determined by $f\in L^2(X,\mu)$, denoted by $Z(f)$ is the closure of the linear span of $\{U^n(f):n\in\Z\}$. A cyclic subspace is \emph{maximal} if it is not contained in any larger cyclic subspace.
The spectral theorem for unitary operators allows us to define the \emph{spectral type} $\sigma_T$, which is (up to equivalence) the spectral measure determined by any vector determining a maximal cyclic subspace. The spectral type $\sigma_T$ can be written as $\sigma_T=\sigma_d + \sigma_c$, where $\sigma_d$ is purely atomic measure and $\sigma_c$ is a continuous (non-atomic) measure.

The \emph{point spectrum} of a measure-preserving transformation $(X,T,\mu)$ is the set of eigenvalues for the Koopman operator $U_T(f)=f\circ T$, i.e.
\[
H(T,\mu)=\{\lambda \in\mathbb{C}: \lambda f=f\circ T \text{ for some }f\in L^2(X,\mu)\text{ with }f\neq 0\}.
\]
The point spectrum is a subset of the unit circle $\mathbb{S}^1\subset \C$. Since $L^2(X,\mu)$ is separable, the point spectrum is at most countable and, in the ergodic case, it forms a multiplicative subgroup of $\mathbb{S}^1$. Furthermore, the spectral measure of each eigenvector is discrete. The purely atomic part of the spectral type is closely connected with 
spectral measures associated to eigenvectors; in particular, $\sigma_d$ is supported on $H(T,\mu)$.


Let $(X,T,\mu)$ and $(Y,S,\nu)$ be measure preserving systems.
A \emph{joining of $(X,T,\mu)$ and $(Y,S,\nu)$} is a $T\times S$-invariant measure $\eta$ on $X\times Y$ with marginals $\mu$ and $\nu$.
We write $J(\mu,\nu)$ for the set of all joinings of $(X,T,\mu)$ and $(Y,S,\nu)$.
We say that $(X,T,\mu)$ and $(Y,S,\nu)$  are \emph{disjoint} if $\mu\times\nu$ is the only joining of these two systems.

We will also use in the proofs of our theorems the description of the spectrum for rotations and the characterization of the ergodicity of the product of two transformations via their spectra. If $R_\alpha$ is an irrational rotation, then the spectrum $H(R_\alpha,\lambda)$ with respect to the invariant Lebesgue measure $\lambda$ is generated by
$e^{2\pi i \alpha}$, i.e.
$$H(R_\alpha,\lambda)=\{\exp(2\pi i k\alpha):k\in\Z\}.$$
Furthermore, the spectral type of the Koopman operator associated with $(\cir, R_\alpha,\lambda)$ is purely atomic.

\begin{lem}\label{lem:product_erg}
The product of two ergodic measure-preserving systems $(X,T,\mu)$ and $(Y,S,\nu)$ is ergodic if and only if their point spectra have trivial intersection, i.e. $H(T,\mu)\cap H(S,\nu)=\{1\}$.
\end{lem}
This fact might be not a part of classical textbooks on ergodic theory, but can be found e.g. in \cite{Parry}, page 46, Exercise 2. In similar spirit, we have the following criterion for disjointness.

\begin{thm}[Thm. 6.28 in \cite{GlasnerBook}]\label{thm:Glasner-disjoint}
If  $(X,T,\mu)$ and $(Y,S,\nu)$ are measure preserving systems whose spectral types are singular except for the common atom at $1$, 
then $(X,T,\mu)$ and $(Y,S,\nu)$ are disjoint.
\end{thm}

The following lemma was suggested to us by Lema\'{n}czyk and Przytycki.

\begin{lem} \label{lem:LemPrzy}
Let $\alpha\in\R\setminus\Q$.
The measure-preserving systems $(\cir, R_\alpha, \lambda)$ and $(X,\mu,T)$ are disjoint if (and only if) $\exp(2\pi i k\alpha)\notin H(T,\mu)$ for $k\in\Z \setminus \{0\}$.
\end{lem}
\begin{proof}

 It follows from the assumptions that $H(X,\mu,T)\cap H(\cir, R_\alpha, \lambda)=\{1\}$. 
 Since the spectral type of $R_\alpha$ is purely atomic,
and the purely atomic part of the spectral type of $T$ is supported on $H(X,T,\mu)$, 
the spectral types of $(X,T,\mu)$ and $(\cir, R_\alpha, \lambda)$
are singular (except the common atom at $1$). By Theorem \ref{thm:Glasner-disjoint}, this implies that $(\cir, R_\alpha, \lambda)$ and $(X,T,\mu)$ are disjoint.
\end{proof}

\begin{thm} \label{thm:disjoint-rotation}
For every ergodic measure-preserving system $(X,T,\mu)$ there exists an irrational $\alpha$ such that $(\cir, R_\alpha, \lambda)$ and $(X,\mu,T)$ are disjoint.
\end{thm}
\begin{proof}
The point spectrum of $(X,\mu,T)$ is a countable subgroup of the circle, hence there is an $\alpha\in\R\setminus\Q$ such that $\exp(2\pi i k\alpha)\notin H(T,\mu)$ for all $k\in\Z\setminus\{0\}$. It remains to apply Lemma \ref{lem:LemPrzy}.
%
\end{proof}

\section{Arcwise connectedness}
In this section we prove our main result. But first we prepare some notation and state an auxiliary lemma.

Let $Z$ be a topological space and $x,y\in Z$.
A \emph{path} (resp.\ \emph{arc}) from $x$ to $y$ in $Z$ is a continuous function (resp.\ homeomorphism onto the image) $\gamma\colon[0,1]\to Z$ such that $\gamma(0)=x$ and $\gamma(1)=y$.
The space $Z$ is \emph{pathwise connected} (\emph{arcwise connected}) if for every $x,y\in Z$ there is a path (an arc) from $x$ to $y$.

By $\dbar(A)$ we denote the \emph{upper asymptotic density} of a set $A\subset \Z$. Recall that
\[
\dbar(A)=\limsup_{n\to\infty}\frac{|A\cap\{1,\ldots,n\}|}{n}.
\]
Given $x,y\in \Omega=\Lambda^\Z$ with $x=(x_j)_{j\in\Z}$ and $y=(y_j)_{j\in\Z}$ the formula
\[
\dbarom(x,y)=\dbar\left(\{n\in\N\mid x_n\neq y_n\}\right).
\]
defines a pseudometric on $\Omega$. Note that $\dbarom(x,y)=0$ implies that $x$ and $y$ differ on a set of coordinates of zero upper asymptotic density. The $d$-bar pseudometric $\dbarom$ is closely connected with a metric on the set of shift-invariant measures $\Omega$ stronger than the usual metric determining the weak$^*$ topology. This metric, denoted by $\dbarm$ is defined for $\mu,\nu\in\Ms(\Omega)$ by
\[
\dbarm(\mu,\nu)=\inf_{\eta\in J(\mu,\nu)} \eta(\{(x,y)\in\Omega\times\Omega: x_0\neq y_0\}),
\]
where as in the previous section $J(\mu,\nu)$ denotes the set of all joinings of $\mu$ and $\nu$ (see \cite[Theorem I.9.7]{Shields}.
The link between the $d$-bar metric $\dbarm$ on measures and the pseudometric $\dbarom$ on generic points needed here is the following corollary of Lemma I.9.8 in \cite{Shields}:
\begin{equation}\label{ineq:Shields}
\dbarm(\bmu(x),\bmu(y))\leq\dbarom(x,y),\qquad x,y\in\Gen(\Omega,\sigma).
\end{equation}

For $x\in \Omega$ and $y\in\{0,1\}^\Z$ we denote by $x \mult y$ the coordinate-wise product: $(x \mult y)_j = x_j \cdot y_j$.
If $X\subset\Omega$ and $Y\subset\{0,1\}^\Z$ are shift spaces, then we denote the image of $X\times Y\subset\Omega$ through $\mult$ as $X\mult Y$. Observe that $0$ is a safe symbol for a shift space $X$ over $\{0,1,\ldots,n-1\}$ if and only if $X\mult \{0,1\}^\Z=X$. It is clear that $\mult\colon X\times Y\to X\mult Y$ is a factor map of $(X\times Y,\sigma\times\sigma)$ onto $(X\mult Y, \sigma)$. It follows that if $(x,y)\in\Gen(X\times Y,\sigma\times\sigma)$, 
then $x\mult y\in\Gen(X\mult Y,\sigma)$. 
Unfortunately, it is not always the case that if $x$ and $y$ are generic points, then $(x,y)$ is generic, see \cite[p. 22]{Furstenberg}.  But it holds if we assume that the corresponding measures are disjoint. The proof follows the same lines as the proof of Theorem I.6 in \cite{Furstenberg}.

\begin{lem}\label{lem:disjoint=>generic}
If  $(X,T,\mu)$ and $(Y,S,\nu)$ are disjoint measure preserving systems, $x\in\Gen_T(\mu)$ and $y\in\Gen_S(\nu)$ then $(x,y)\in\Gen_{T\times S}(\mu\times\nu)$.
\end{lem}

Finally, given a shifts space and $t\ge 0$ we write $\Mst(X)$ for the set of ergodic measures on $X$ with entropy less or equal than $t$, that is, $\Mst(X)=\{\mu\in\Mse(X):h(\mu)\le t\}$.
This was the last piece we needed for a proof of our main result.

\begin{thm}\label{thm:arcwise}
If $X$ is a shift space with a safe symbol (in particular, if $X$ is a hereditary shift)  and $t\ge 0$, then $\Mst(X)$ endowed with $\dbarm$ is arcwise connected.
\end{thm}


\begin{proof} Without loss of generality we assume that $0$ is a safe symbol for $X$. 
By \cite[Cor. 31.6]{Willard} every pathwise connected Hausdorff space is also arcwise connected. Therefore it suffices to show that for any $\mu \in \Mse(X)$ there exists a path from $\mu$ to the Dirac measure $\delta_\mathbf{0}$, where $\mathbf{0}$ denotes the bi-infinite sequence of $0$'s  That is, we need to define a continuous function $\hat\Phi\colon[0,1]\to\Mse(X)$ with  $\hat\Phi(0)=\delta_\mathbf{0}$ and $\hat{\Phi}(1)=\mu$ which is continuous when we endow $\Mse(X)$ with the $\dbarm$ metric. To do that we construct a $\dbarom$-continuous path of generic points.
Since the topology introduced by $\dbarom$ on $X$ is not Hausdorff 
we consider an equivalence relation on $X$ defined by $x\equiv y$ if $\dbarom(x,y)=0$. The resulting set of equivalence classes endowed with the metric induced by $\dbarom$ is called the \emph{Besicovitch space} of $X$ and denoted $X_B$. It is easy to see that a $\dbarom$ continuous path in $X$ leads to a path in $X_B$ and the arc in $X_B$ leads to a $\dbarom$ continuous arc in $X$. Hence it remains to define a $\dbarom$-continuous path of generic points of ergodic measures in $X$. 

For $\alpha,\beta \in [0,1]$, define the point $y_{\alpha,\beta} \in \Omega$ by
$$
	y_{\alpha, \beta} =  \left( \chi_{[0,\beta)}( j \alpha \bmod 1) \right)_{j \in \Z} \in  \Omega,
$$
where $\chi_{[0,\beta)}$ denotes the characteristic function of $[0,\beta)$. Note that $y_{\alpha,0}=\mathbf{0}$ and $y_{\alpha,1}=\mathbf{1}$, where $\mathbf{1}$ denotes the bi-infinite sequence of $1$'s. Write $Y_{\alpha,\beta}$ for the closure of the orbit of $y_{\alpha,\beta}$ with respect to $\sigma$.

It is a well known fact that if $\alpha$ is irrational and $0 < \beta < 1$, then  $(Y_{\alpha,\beta}, \sigma)$ is isomorphic (in the category of measure preserving systems) to the circle rotation $(\cir,R_\alpha)$. Furthermore it is minimal and has unique invariant measure, which we denote by $\nu_{\alpha,\beta}$. In addition, let $\nu_{\alpha,0}=\delta_\mathbf{0}$ be the Dirac measure concentrated on $\mathbf{0}$ and similarly $\nu_{\alpha,1}=\delta_\mathbf{1}$. With this notation $y_{\alpha,\beta}$ is generic for $\nu_{\alpha,\beta}$ for any $\beta\in[0,1]$.
Therefore, for $\beta\in(0,1)$, measure preserving systems $(Y,S,\nu)$ and $(Y_{\alpha,\beta}, \sigma, \nu_{\alpha,\beta})$ are disjoint if and only if $(Y,S,\nu)$ and $(\cir,R_\alpha,\lambda)$ are disjoint.

Fix the choice of an irrational $\alpha \in [0,1]$ such that $(X,\sigma,\mu)$ and $(\cir,R_\alpha,\lambda)$ are disjoint, whose existence is assured by Theorem \ref{thm:disjoint-rotation}. By Lemma \ref{lem:disjoint=>generic}, for any $\mu$-generic point $x \in X$ and any $\beta \in [0,1]$, the point $(x, y_{\alpha,\beta})$ is generic for the ergodic measure $\mu \times \nu_{\alpha,\beta}$. It follows that $x*y_{\alpha,\beta}$ is a generic point for some ergodic measure $\mu_\beta=\bmu(x*y_{\alpha,\beta})$ for $\beta\in[0,1]$.
Furthermore, $x*y_{\alpha,0}=\mathbf{0}$ and $x*y_{\alpha,1}=x$, hence $\mu_0=\delta_\mathbf{0}$ and $\mu_1=\mu$. Note that $h(\mu \times \nu_{\alpha,\beta})=h(\mu)$, thus $h(\mu_\beta)\le h(\mu)\le t$ for all $0\le \beta\le 1$.

For $0\le \beta<\beta'\le 1$ we have
$(y_{\alpha,\beta})_j\neq (y_{\alpha,\beta'})_j$ if and only if  $R^j_\alpha(0)\in [\beta,\beta')$.
Therefore
\[
\dbarom(y_{\alpha,\beta},y_{\alpha,\beta'})=\dbar(\{j\in\Z:(y_{\alpha,\beta})_j\neq (y_{\alpha,\beta'})_j\})=\dbar(\{j\in\Z: R^j_\alpha(0)\in [\beta,\beta')\}).
\]
Since $R_\alpha$ is a uniquely ergodic transformation, we have
\[
\dbar(\{j\in\Z: R^j_\alpha(0)\in [\beta,\beta')\})=\lambda([\beta,\beta'))=\beta'-\beta.
\]
It follows that the map
$ \Phi \colon[0,1]\ni\beta\mapsto x*y_{\alpha,\beta}\in\Omega$ 
is $\dbarom$-continuous, because
\[
\dbarom(x*y_{\alpha,\beta},x*y_{\alpha,\beta'})
\le \dbarom(y_{\alpha,\beta},y_{\alpha,\beta'})
=|\beta-\beta'|.
\]
Furthermore since $0$ is a safe symbol for $X$ we have $\Phi(X)\subset X$.
 Hence, by \eqref{ineq:Shields}, the map $\hat{\Phi}\colon[0,1]\ni \beta \mapsto \mu_\beta = \hat \mu \circ \Phi(\beta)\in\Mst(X)\subset \Mse(X)$
is also $\dbarm$-continuous, and establishes a path from $\hat\Phi(0)=\bmu(\mathbf{0}) = \delta_\mathbf{0}$ to $\hat\Phi(1)=\bmu(x) = \mu$ in $\Mst(X)$.
\end{proof}

By \cite[Theorem 7.7]{Rudolph} the topology of $\dbarm$ metric is stronger than the weak$^*$ topology on $\Mse(\Omega)$. Together with Theorem \ref{thm:arcwise} it yields.

\begin{cor}\label{cor:weak-star-arc-connectedness}
If $X$ is a shift space with a safe symbol (in particular, if $X$ is a hereditary shift), then for any $t\ge 0$ the set $\Mst(X)$ (in particular, $\Mse(X)$) is arcwise connected in the weak$^*$ topology.
\end{cor}


As the entropy function on $\Mse(\Omega)$ endowed with the $\dbarm$ metric is continuous (see \cite[Theorem 7.9]{Rudolph} or \cite[Theorem I.9.16]{Shields}) we conclude also the following

\begin{cor}\label{cor:entropy-interval}
If $X$ is a shift space with a safe symbol (in particular, if $X$ is a hereditary shift),
then $\{h(\mu):\mu\in\Mse(X)\}=[0,h_{\text{top}}(X)]$ (possibly degenerate to a point).
\end{cor}

If $X$ is a hereditary shift, then the bi-infinite sequence of $0$'s denoted by $\mathbf{0}$ is a fixed point for the shift map and belongs to $X$. Hence the atomic measure $\delta_\mathbf{0}$ carried by $\mathbf{0}$ is invariant for $X$. There are hereditary shifts for which $\delta_\mathbf{0}$ is the only invariant measure and the existence of another invariant measure has many consequences (see \cite{hh4} for more details). Thus we divide all hereditary shifts into two disjoint classes:
\begin{enumerate}[I.]
  \item \emph{uniquely ergodic} hereditary shifts  ($\delta_\mathbf{0}$ is the unique invariant measure), \label{class:I}
  \item \emph{non-uniquely ergodic} hereditary shifts. \label{class:II}
\end{enumerate}
Hereditary shift spaces in class \eqref{class:I} are characterized as those in which for every point $x\in X$ the symbols other than $0$ appear in $x$ on a set of coordinates of zero asymptotic density. Class \eqref{class:I} coincides also with the hereditary shifts having zero topological entropy.
Although this is not stated explicitly in \cite{hh4} the proof applies verbatim to shift spaces with a safe symbol. Hence we may note the following corollary.

\begin{cor}\label{cor:entropy-positive}
A shift space with a safe symbol (in particular, a hereditary shift) has positive topological entropy if and only if it has uncountably many ergodic measures.
\end{cor}

\section{Examples of hereditary shifts}\label{sex:examples}

Here we list some notable examples of hereditary shifts to which our main result can be applied.

The primary example of a $\mathscr{B}$-free shift is  the \emph{square-free shift} considered by Sarnak \cite{Sarnak}; that is a shift space, whose structure reflects the statistical properties of square-free numbers. Recall that $n\in\N$ is \emph{square-free} if there is no prime number $p$ such that $p^2$ divides $n$. Let $\eta$ be the characteristic function of the square-free numbers treated as point in $\{0,1\}^\Z$. The \emph{square-free shift} is the closure of the orbit of $\eta$ with respect to the shift map and it turns out it is hereditary \cite{Peckner,Sarnak}. The study of the square-free shift has been recently extended \cite{ALR,BKKPL} to the general $\mathscr{B}$-free shifts $X_\bb$ induced in the same manner by the characteristic function $\bb$ of $\mathscr{B}$-free numbers, that is, integers with no factor in a given
set $\mathscr{B}\subset\N$. If $\mathscr{B}$ is an Erd\H{o}s set, that is, it consists of pairwise co-prime integers and $\sum_{b\in\mathscr{B}}1/b<\infty$, then the $\mathscr{B}$-free shift is also hereditary, see \cite{ALR}. In general this is not the case, but the smallest hereditary shift containing $X_\bb$ still has some interesting properties (see \cite{BKKPL}). These systems were also investigated by Avdeeva \cite{Avdeeva}, Cellarosi and Sinai \cite{CS}, Ku{\l}aga-Przymus,
Lema\'{n}czyk, and Weiss \cite{KPLW,KPLW2}, Peckner \cite{Peckner}. 

Another shift space related to $\BB$-free integers is the $\BB$-admissible shift \cite{BKKPL}. We say that a sequence $x=(x_j)_{j\in\Z}\in \{0,1\}^\Z$ is  \emph{$\BB$-admissible} if for every $b\in\BB$ the set $\{ j\in\Z: x_j=1\}$ is disjoint with a set $b\Z+r$ for some $0\le r<b$. It is not hard to see that the set of $\BB$-admissible sequences in $\Omega$ is a hereditary shift space $X_\BB$ called the \emph{$\BB$-admissible shift}.
Because $\bb$ is clearly a $\BB$-admissible sequence we see immediately that $X_\bb\subset X_\BB$ and the equality holds if $\BB$ is an Erd\H{o}s set.

Beta shifts introduced by R\'{e}nyi \cite{R} are related to number theory, tilings, and dynamics
of discontinuous transformations. For $\beta>1$ the \emph{beta shift} $\Omega_{\beta}$ is the closure of
the set of sequences in $\{0, 1 ,\ldots, \lfloor\beta\rfloor\}^\mathbb{N}$ arising as greedy $\beta$-expansions of numbers from $[0,1]$.
All beta shifts are hereditary (see \cite{hh4}).

Spacing shifts were introduced by Lau and Zame in \cite{LZ}. A \emph{spacing shift} $\Omega_P$,
where $P\subset\mathbb{N}$, is
the set of all $x=(x_i)\in\{0,1\}^\Z$ such
that $x_i=x_j=1$ and $i\neq j$ imply $|i-j|\in P$.
Spacing shifts were studied in \cite{B,spacing}. It is easy to see that they are hereditary.

Bounded density shifts were recently introduced by Stanley \cite{Stanley}.
They are defined by fixing a function $f\colon\N\to[0,\infty)$
and considering the set of all biinfinite sequences such that for each $p\in\N$ the sum of the entries of any finite subword of length $p$ do not exceed $f(p)$. Since for any word coordinatewise smaller than a given word the sum of the entries can only decrease, these shifts are hereditary.

\section{Intermediate entropy property vs arcwise connectedness}

The following results show that the conclusions of Corollary \ref{cor:weak-star-arc-connectedness} and Corollary \ref{cor:entropy-interval} are independent of each other\footnote{We are grateful to Tomasz Downarowicz for drawing our attention to the theory presented in \cite{D-book} and generously sharing his insight on these matters.}.
Recall that for every dynamical system $(X,T)$ the set $\MT(X)$ endowed with the weak$^*$ topology has the structure of a Choquet simplex (see \cite{D-book}). We say that a nonempty metrizable convex compact subset $K$ of a locally convex topological vector space is a \emph{Choquet simplex} if every point of $K$ is the barycenter of a unique probability measure supported on the set $\ext K$ of extreme points of $K$.

\begin{lem}\label{lem:convex}
Let $K$ be a convex subset of a vector space (over $\R$ or $\C$).
If $P\subset\ext K$ is closed, then its characteristic function $\chi_P$ is convex, that is, for $x,y\in K$ and $0<\alpha<1$ we have
$\chi_P(\alpha x+(1-\alpha)y)\le \alpha\chi_P(x)+(1-\alpha)\chi_P(y)$.
\end{lem}
\begin{thm}\label{thm:DS}
For every Polish topological space $P$ there exists a minimal shift space $X$ (a Toeplitz shift) such that $\Mse(X)$ with the weak$^*$ topology and $P$ are
 homeomorphic  and there is a unique measure $\mu\in\Mse(X)$ with positive metric entropy. (In particular, if $P$ has more than one point, then 
 $X$ does not have the intermediate entropy property.)
\end{thm}

\begin{proof}
Let $K_P$ be the Choquet simplex whose set of extreme points $\ext K_P$ is homeomorphic to $P$ (such a Choquet simplex exists by \cite{H}). 
Fix an extreme point $z \in \ext K_P$. Let $\delta_z$ denote the Dirac measure concentrated on $\{z\}$ and let $\chi_{\{z\}}$ denote the characteristic function of $\{z\}\subset\ext K_P$. Given a point $x\in K_P$, we denote by $\xi^x$ the unique probability measure concentrated on $\ext K_P$ such that $x$ is the barycenter of $\xi^x$, that is, $x$ is equal to the Pettis integral of the identity with respect to $\xi^x$:
\[
x=\int_{K_P}y\,d\,\xi^x(y).
\]
(See Appendix in \cite{D-book} for more details.) We define a function $\varphi \colon K_P\to \R$ by
\[
	\varphi(x)=\int_{\ext K_P} \chi_{\{z\}}(\lambda)\xi^x(d\lambda)=\xi^x(\{z\})\,.
\]	
This function is the harmonic prolongation of $\chi_{\{z\}}$ \cite[Definition A.2.18]{D-book}.
The characteristic function of a closed set is upper semi-continuous. Furthermore it is convex by Lemma \ref{lem:convex}. 
By \cite[Fact A.2.10]{D-book} every upper semi-continuous harmonic function on a Choquet simplex is affine.
Therefore, $\varphi$ is bounded, affine, non-negative, upper semi-continuous on $K_P$, and $\varphi|_{\ext K_P}=\chi_{\{z\}}$.
\newcommand{\psitmp}{\psi}

By~\cite[Theorem 1]{DS}, there exists a minimal Toeplitz shift $(X,T)$ and an affine
(onto) homeomorphism $\psitmp\colon K_P \to \MT(X)$, such that for every $x \in K_P$,
$\varphi(x) =h(\psitmp(x))$,  where $h$ denotes the entropy function.
This proves the proposition with $\mu = \varphi(z)$.
\end{proof}
\begin{thm}\label{thm:DS2}
For every uncountable Polish topological space $P$ there exists a minimal shift space $X$ (a Toeplitz shift) such that $\Mse(X)$ with the weak$^*$ topology and $P$ are 
 homeomorphic  and the metric entropy function $h$ restricted to $\Mse(X)$ is not constant and has the Darboux (intermediate value) property.
\end{thm}
\begin{proof}
Assume that $P$ is uncountable. Then there is $C\subset P$ homeomorphic to the usual Cantor set in $[0,1]$ \cite[Theorem 3.2.7]{Srivastava}. 
Let $g\colon P\to[0,1]$ be any continuous function such that $g|_C$ maps $C$ onto $[0,1]$ (say, use the  Tietze extension theorem to extend Devil's staircase function on $C$ to the whole $P$). Then take $\eta=g\cdot \chi_C$ where $\chi_C$ is the characteristic function of $C$ in $P$. The function $\eta\colon P\to [0,1]$ is upper semi-continuous and has the Darboux (intermediate-value) property. Reasoning as above, we can find a Toeplitz minimal shift $(X,T)$ such that (up to affine homeomorphism) $\MT(X)=K_P$ and $\eta=h|_{\MTe(X)}$.
\end{proof}

\section{Final remarks}



We find the following question intriguing: Let $\mu$ denote the M\"{o}bius function extended to $\Z$ in an obvious way (say $\mu(0)=$ and $\mu(-n)=\mu(n)$ for $n\in\N$). Let $X_M$ be the orbit closure of $\mu$ in $\{-1,0,1\}^\Z$. Is $0$ the safe symbol for $X_M$?
Recall that the study of the square-free system in \cite{Sarnak} was motivated by questions about M\"{o}bius function $\mu$.

Some special cases of our results were known to hold for some time. It follows from the general theory of Choquet
simplices that if the set of ergodic measures is dense in the simplex of all invariant measures in the weak$^*$ topology, then it is arcwise connected set \cite{LOS}. Therefore  the set of ergodic measures is arcwise
connected in the weak$^*$ topology for all hereditary $\BB$-free shifts $X_\bb$, as $\overline\Mse(X_\bb)=\Ms(X_\bb)$ by
\cite{BKKPL,KPLW2}. For a concrete example of a shift space with dense set of ergodic measures but without the intermediate entropy property see \cite{GK}.
Ku{\l}aga-Przymus, Lema\'{n}czyk and Weiss
showed also in \cite{KPLW2} that ergodic measures need not be dense among all invariant measures for a general hereditary system, but must be arcwise connected by our result. This phenomenon was previously observed in the Dyck shift by Climenhaga \cite{VC}.

All the results stated in this work remain valid for unilateral shift spaces, that is, closed $\sigma$-invariant subsets of $\Lambda^\N$. This follows directly from \cite[Proposition 2.1]{CT}.

After we have finished writing this paper Ay\c{s}e \c{S}ahin kindly shared
with us the article \cite{Ayse}, where $\dbarom$ continuous arcs of ergodic
measures are constructed for some examples of $\Z$ and $\Z^2$ shift
spaces. Anthony Quas observed that one can use the construction from \cite{QS} (loosely related to the ``grand coupling'' that shows up
in probability theory) to get another proof of Theorem \ref{thm:arcwise} for the case $t=h_\text{top}(X)$, that is, for $\Mse(X)$. This approach does not work for smaller $t$. The reader interested in the result about the measures only, can also use the idea from \cite{Ayse} and coupling with irrational rotations to simplify the proof of Theorem \ref{thm:arcwise} at the cost of loosing the explicit construction of generic points. The latter result is in the spirit of ``single orbit dynamics'' (see remarks about differences between almost everywhere results and single orbit theorems in \cite[pages 8 and 89]{WeissBook}) and we hope the reader will also find it interesting. Besides it gives some information about topological properties of the Besicovitch space of the shift (see the proof of Theorem \ref{thm:arcwise}). We are grateful to Anthony Quas for sharing his idea with us.

\appendix

\section{Uniquely ergodic extensions}\label{sec:appendix-A}
Originally, we proved Theorem \ref{thm:arcwise} using tools developed in this section. After a talk of one of us presented these results during a seminar at the Institute of Mathematics of the Polish Academy of Sciences, professors Lema\'{n}czyk and Przytycki kindly suggested another approach for this proof. As we feel that the following observations might be of the independent interest and allow us to show Theorem \ref{thm:arcwise} avoiding spectral theory of unitary operators we decided to attach them here.

Suppose that $(X,T)$ is a topological factor of $(Y,S)$ and $\pi\colon Y\to X$ is a factor map.
Let $\mu\in\MTe(X)$. We say that $(Y,T)$ is a \emph{uniquely ergodic extension} of $(X,T)$ over $\mu$ if there is a unique measure $\nu\in\mathcal{M}^e_S(Y)$ such that $\pi_*(\nu)=\mu$. We also say that $(Y,S,\nu)$ \emph{uniquely extends} $(X,T,\mu)$ through $\pi\colon Y\to X$.
The following lemma and its corollary explain our interest in uniquely ergodic extensions.

\begin{lem}[Weiss, \cite{WeissBook} Proposition 3.4]\label{lem:Weiss} 
If the ergodic system $(Y,S,\nu)$ uniquely extends $(X,T,\mu)$ over $\pi\colon Y\to X$, $x_0\in \Gen_T(\mu)$ and $y_0\in\pi^{-1}(\{x_0\})$, then $y_0\in \Gen_{S}(\nu)$.
\end{lem}

The lemma above implies a useful criterion for genericity of a pair of generic points with respect to a product measure.

\begin{cor} \label{cor:Weiss}
Let $(X,T,\mu)$ and $(Y,S,\nu)$ be two ergodic systems with ergodic product.
If $(X\times Y,T\times S,\mu\times \nu)$ uniquely extends $(X,T,\mu)$ over the projection onto the first coordinate $\pi\colon X\times Y\to X$,
then for every $x_0\in \Gen_T(\mu)$ and $y_0 \in Y$ the pair $(x_0,y_0)$ is generic for $\mu\times \nu$.
\end{cor}

Note that in the situation of the above corollary, $(Y,S,\nu)$ in necessarily uniquely ergodic, so $\Gen_S(\nu) = Y$. Lemma \ref{lem:Weiss} was used in place of Lemma \ref{lem:disjoint=>generic} in the first version of our proof of Theorem \ref{thm:arcwise}.

The following relative version of the unique ergodicity theorem of Furstenberg can be used to show that certain extensions are uniquely ergodic over group rotations. The proof can be obtained through an easy modification of the standard argument which can be found in e.g.\ \cite[Theorem 4.21]{EinsiedlerWardBook}; we include it here for the convenience of the reader.


\begin{thm}\label{thm:Furstenberg}
Assume that $T\colon X\to X$ is a homeomorphism of a compact metric space and $\mu\in\MTe(X)$. Let $G$ be a compact group, $\lambda_G$ be the Haar measure for $G$ and $\varphi\colon X\to G$ be a continuous map. Define $Y=X\times G$ and $S\colon Y \to Y$ by $S(x,g)=(T(x),\varphi(x)g)$. If $S$ is ergodic with respect to $\mu\times\lambda_G$, then $(Y,S,\mu\times\lambda_G)$ uniquely extends $(X,T,\mu)$ over the projection $\pi\colon Y\ni(x,g)\mapsto x\in X$.
\end{thm}
\begin{proof}
The $S$-invariance of $\mu \times \lambda_G$ is an immediate consequence of Fubini's theorem, since for any $f \in \mathcal{C}(Y)$ we have
\begin{align*}
	\int_Y Sf \;\text{d} (\mu \times \lambda_G)
	&= \int_X \int_G f(Tx, \phi(x)g) \;\text{d} \lambda_G(g) \;\text{d} \mu(x)
	\\&= \int_G \int_X f(Tx, g) \;\text{d} \mu(x) \;\text{d} \lambda_G(g)
	= \int_Y f \;\text{d} (\mu \times \lambda_G).
\end{align*}

Assume that $(Y,S,\mu\times\lambda_G)$ is an ergodic measure preserving system. We claim that for every $h\in G$ the set $\Gen_S(\mu\times\lambda_G)$ is invariant under the map $R_h\colon Y\ni(x,g)\mapsto (x,gh)\in Y$. Note that $(x,g)\in \Gen_S(\mu\times\lambda_G)$ if and only if for every continuous function $f\colon X\to\mathbb{R}$ we have
\begin{equation}\label{eq:generic}
\lim_{N\to\infty}\frac{1}{N}\sum_{n=0}^{N-1} f(S^n(x,g))=\int_Y f \;\text{d}(\mu\times\lambda_G).
\end{equation}
We fix $f\in\mathcal{C}(Y)$ and $h\in G$.
We want to show that \eqref{eq:generic} holds with $(x,gh)$ in place of $(x,g)$.
First note that
\begin{equation}\label{eq:ergodic-av}
\lim_{N\to\infty}\frac{1}{N}\sum_{n=0}^{N-1} f(S^n(x,gh))=\lim_{N\to\infty}\frac{1}{N}\sum_{n=0}^{N-1} f\circ R_h(S^n(x,g)).
\end{equation}
Now use the fact that $(x,g)$ is generic to conclude that the right hand side of \eqref{eq:ergodic-av} converges to
\[
\int_Y f\circ R_h \;\text{d}(\mu\times\lambda_G)=\int_Y f \;\text{d}((R_h)_*(\mu\times\lambda_G))=\int_Y f \;\text{d}(\mu\times\lambda_G),
\]
where the last equality holds because $\mu\times\lambda_G$ is $R_h$ invariant. This shows that for all $h\in G$ we have
$
R_h(\Gen_S(\mu\times\lambda_G))\subset \Gen_S(\mu\times\lambda_G).
$
 Repeating the same argument with $h^{-1}$ in place of $h$ yields the reverse inclusion, and hence for any $h \in G$ we have
\[R_h(\Gen_S(\mu\times\lambda_G))= \Gen_S(\mu\times\lambda_G).\]
Therefore we can find a set $E_1\subset X$ such that $\mu(E_1)=1$ and $\Gen_S(\mu\times\lambda_G)=E_1\times G$.

Let $\nu$ be an ergodic $S$-invariant measure such that $\pi_*(\nu)=\mu$. Let $E_2=\pi(\Gen_S(\nu))$. Then $\mu(E_2)=1$, hence $\mu(E_1 \cap E_2) = 1$ and in particular
there is $x_0\in E_1\cap E_2$. Let $y_0\in \Gen_S(\nu)$ be such that $\pi(y_0)=x_0$. But then $y_0\in \{x_0\}\times G\subset \Gen_S(\mu\times \lambda_G)$, and
$y_0$ is generic for both $\nu$ and $\mu\times\lambda_G$, implying that $\nu=\mu\times\lambda_G$.
\end{proof}

Theorem \ref{thm:Furstenberg} yields the following corollary which may replace Theorem \ref{thm:disjoint-rotation} in the proof of our main result.

\begin{cor}\label{cor:product-ergodic}
If $T\colon X\to X$ is a homeomorphism of a compact metric space and $\mu\in\MTe(X)$, then there exists $\alpha\in\R\setminus\mathbb{Q}$ such that the product system
$(T\times R_\alpha, X\times\cir,\mu\times\lambda)$ is a uniquely ergodic extension of $(X,T,\mu)$ over the projection $\pi\colon X\times\cir \to X$, where $R_\alpha$ is the rotation by $\alpha$ and $\lambda$ denotes the Lebesgue measure on $\cir$ .
\end{cor}
\begin{proof}
By Theorem \ref{thm:Furstenberg}, it is enough to find $\alpha\in\R\setminus\mathbb{Q}$ such that $T\times R_\alpha$ is ergodic with respect to $\mu\times\lambda$. To this end, using spectral theory discussed in Section \ref{sec:spectral}, it will suffice to ensure that only common element of $H(T,\mu)$ and $H(R_\alpha, \lambda) = \{ \exp(2 \pi i k \alpha : k \in \Z\}$ is $1$. Because $H(T,\mu)$ is at most countable, this will be the case for all but countably many choices of $\alpha$.
\end{proof}



\end{document}